\documentclass[12pt,reqno]{amsart}
\usepackage{fullpage}   
\usepackage{amssymb}    
\usepackage{amsmath}
\usepackage{amsthm}
\usepackage{array}
\usepackage{enumitem}
\usepackage{url}
\usepackage{verbatim}
\usepackage{xcolor}
\usepackage{mathtools}
\theoremstyle{plain}
\newtheorem{theorem}{Theorem}
\newtheorem{lemma}[theorem]{Lemma}

\newtheorem{proposition}[theorem]{Proposition}
\newtheorem{corollary}[theorem]{Corollary}

\theoremstyle{remark}
\newtheorem{example}[theorem]{Example}

\newtheorem*{acknowledgment}{Acknowledgment}

\newcommand{\ldiv}{\backslash}
\newcommand{\rdiv}{/}

\newcommand{\Sym}{\mathrm{Sym}}

\newcommand{\Aut}{\mathrm{Aut}}
\newcommand{\Semi}{\mathrm{SemiAut}}
\newcommand{\Mlt}{\mathrm{Mlt}}
\newcommand{\Inn}{\mathrm{Inn}}
\newcommand{\A}{f^{-m}(x^{-1})}

\title{Moufang semidirect products of loops with groups and inverse property extensions}

\author[M. Greer]{Mark Greer}
\email{\url{mgreer@una.edu}}

\author[L. Raney]{Lee Raney}
\email{\url{lraney@una.edu}}

\address{Department of Mathematics\\
One Harrison Plaza \\
University of North Alabama\\
Florence AL 35632 USA }


\begin{document}
\allowdisplaybreaks
\begin{abstract}
We investigate loops which can be written as the semidirect product of a loop and a group, and we provide a necessary and sufficient condition for such a loop to be Moufang.  We also examine a class of loop extensions which arise as a result of a finite cyclic group acting as a group of semiautomorphisms on an inverse property loop.  In particular, we consider closure properties of certain extensions similar to those as in \cite{gagola13}, but from an external point of view.
\end{abstract}

\maketitle

\section{Introduction}
\label{intro}

A \emph{loop} $(Q,\cdot)$ consists of a set $Q$ with a binary operation $\cdot : Q\times Q\to Q$ such that (i) for all $a,b\in Q$, the equations $ax = b$ and $ya = b$ have unique solutions $x,y\in Q$, and (ii) there exists $1\in Q$ such that $1x = x1 = x$ for all $x\in Q$. We denote these unique solutions by $x = a\ldiv b$ and $y = b\rdiv a$, respectively.  Standard references in loop theory are \cite{bruck71, pflugfelder90}. All loops considered here are finite loops.

If $G$, $N$, and $H$ are groups, then $G$ is an \emph{extension} of $H$ by $N$ if $N \trianglelefteq G$ and $G/N \simeq H$.  Extensions of groups are of great interest in group theory. Of particular interest are the ideas of \emph{internal} and \emph{external} semidirect products of groups.  That is, for $G=NH$ with $N\trianglelefteq G$ and $N\cap H=1$, the structure of $G$ is uniquely determined by $N$, $H$, and the action of $H$ on $N$ by conjugation. In this case, $G$ is the internal \emph{semidirect product of $H$ acting on $N$}.  Alternatively, given groups $N$ and $H$ and a group homomorphism $\phi:H \to \Aut(N)$ from $H$ into the group of automorphisms of $N$, one can construct the external semidirect product $G=N\rtimes_{\phi}H$ in a standard way: namely, $G = \{(n,h)\ | \ n \in N, h \in H\}$ with multiplication defined by $(n_1,h_1)(n_2,h_2) = (n_1 \phi(h_1)(n_2), h_1h_2)$. Note that for all $h \in H$, $\phi(h)$ corresponds to conjugation by $h$ in $N$. There is a natural equivalence between 
internal and external semidirect products of groups, and both are usually referred to as \emph{semidirect products}.

Now, let $N$ be a loop and $H$ a group.  Let $\phi:H\to \Sym(N)$ be a group homomorphism from $H$ into the symmetric group on $N$.  Then, the external semidirect product $G=N\rtimes_{\phi}H$ is the quasigroup defined as $G = \{(x,g) \ | \ x \in N, g \in H\}$ with multiplication given by 
\begin{equation}
(x,g)(y,h)=(x\phi(g)(y),gh).
\label{semimult}
\end{equation}
Such semidirect products have been studied in \cite{KJ00}.  Note that if $\phi:H\to \Sym(N)_{1}$, where $\Sym(N)_1$ is the stabilizer of $1$ in $\Sym(N)$, then $G$ is a loop.  In general, properties of the loop $N$ do not necessarily extend to all of $G$. For example, if $N$ is a Moufang loop, $G$ need not be Moufang, even in the case that $H$ acts on $N$ as a group of automorphisms. 

In $\S 2$, we provide a necessary and sufficient condition on $N$ and $\phi$ in order for $G = N \rtimes_\phi H$ to be a Moufang loop (Theorem \ref{t1}). We also provide an example which shows that it is possible to satisfy the hypotheses of Theorem \ref{t1}.

In \cite{gagola13}, the author studies Moufang loops that can be written as the product of a normal Moufang subloop and a cyclic subgroup.  There it is shown that given a Moufang loop $G=NH$, where $N$ is a normal subloop and $H$ is a cyclic subgroup of order coprime to 3, the multiplication in $G$ is completely determined by a particular semiautomorphism of $N$. Conversely, the author notes that an arbitrary extension of a cyclic group of order coprime to 3 by a Moufang loop (with multiplication given as in the conclusion of Theorem \ref{gag}) is not necessarily Moufang.  In $\S 3$, our approach is to consider the \emph{external} extension of a cyclic group $H$ (whose order is not divisible by $3$) by a loop $N$ with multiplication defined as in \cite{gagola13}.  In particular, we show that the class of IP loops is closed under such extensions (Theorem \ref{t2}).  We conclude with examples which illustrate that, in general, that certain classes of loops (\emph{i.e.} diassociative loops, power associative loops, flexible loops, etc.) are not closed under such extensions.
\section{Semidirect products and Moufang loops}
\label{back}
Throughout juxtaposition binds more tightly than an explicit $\cdot$ so that, for instance, $xy\cdot z$ means $(xy)z$. For $x\in Q$, where $Q$ is a loop, define the \emph{right} and \emph{left translations} by $x$ by, respectively, $R_x(y) = yx$ and $L_x(y) = xy$ for all $y\in Q$. The fact that these mappings are permutations of $Q$ follows easily from the definition of a loop.  It is easy to see that $L_{x}^{-1}(y)=x\backslash y$ and $R_{x}^{-1}(y)=y/x$.  We define the \emph{multiplication group of Q}, $\Mlt(Q)=\langle L_{x},R_{x} \mid \forall x\in Q\rangle $ and the \emph{inner mapping group of Q}, $\Mlt(Q)_{1}=\Inn(Q)=\{\theta\in \Mlt(Q)\mid \theta(1)=1\}$, the stabilizer of $1\in Q$.  

A bijection $f :Q\rightarrow Q$ is a \emph{semiautomorphism} of $Q$ if (i) $f(1)=1$ and (ii) $f(x(yx))=f(x)(f(y)f(x))$.  The \emph{semiautomorphism group of Q}, $\Semi(Q)$ is defined as the set of semiautomorphisms of $Q$ under composition. The following proposition will be helpful throughout.

\begin{proposition}
Let $Q$ be a loop with two-sided inverses (\emph{i.e.} $1\backslash x=1/x=x^{-1}$ for all $x\in Q$) and $f\in\Semi(Q)$. Then for all $x\in Q$, $f(x^{-1})=(f(x))^{-1}$
\end{proposition}
\begin{proof}
Since $f$ is a semiautomorphism, $f(x)=f(x\cdot (x^{-1}x))=f(x)(f(x^{-1})f(x))$.  Canceling gives $f(x)^{-1}=f(x^{-1})$.
\end{proof}

\emph{Moufang loops}, which are easily the most studied class of loops, are defined by any one of the following four equivalent identities:
\[(xy)(zx) = x(yz\cdot x) \qquad (xy)(zx) = (x\cdot yz)x\]
\[ (xy\cdot x)z=x(y\cdot xz)  \qquad (zx\cdot y)x=z(x\cdot yx).\]
In particular, Moufang loops are \emph{diassociative} (\emph{i.e.} every subloop generated by two elements is a group).  If $Q$ is Moufang, then $\Inn(Q)\leq \Semi(Q)$ \cite{bruck52}.

The following shows that if $G = N \rtimes_\phi H$ is Moufang, then $H$ acts on $N$ as a group of automorphisms.
\begin{lemma}
Let $N$ be a loop and $H$ a group.  Let $\phi:H\to \Sym(N)_{1}$ be a group homomorphism.  If $G=N\rtimes_{\phi} H$ is Moufang, then $\phi(H)\subseteq Aut(N)$.
\label{l1}
\end{lemma}
\begin{proof}
Since $G$ is Moufang, $(xy)(zx) = x((yz)x)$ for all $x,y,z\in G$.  Therefore for any $h\in H$,
\begin{alignat*}{2}
[(x,1_{H})(1,h)][(z,1_{H})(x,1_{H})]&=(x,1_{H})[(1,h)(z,1_{H})\cdot (x,1_{H})] && \qquad\Leftrightarrow \\
(x\cdot \phi(h)(zx),h)&=(x\cdot \phi(h)(z)\phi(h)(x),h) 
\end{alignat*}
Cancellation gives $\phi(h)(zx)=\phi(h)(z)\phi(h)(x)$.  
\end{proof} 

A loop $Q$ is an \emph{inverse property} loop (or IP loop) if $x$ has a two-sided inverse for all $x\in Q$, and $(yx^{-1})x=y=x(x^{-1}y)$ holds for all $x,y\in Q$.  
\begin{lemma}
Let $N$ be an IP loop.  Then there exists $f\in Aut(N)$ such that 
\begin{equation}\tag{$*$}
(xy)(zf(x))=x(yz\cdot f(x))
\label{eq1}
\end{equation}
for all $x, y, z \in N$ if and only if $N$ is Moufang.
\label{l2}
\end{lemma}
\begin{proof}
Firstly, if $N$ is Moufang, then the trivial isomorphism satisfies \eqref{eq1} for all $x,y,z\in N$.  Conversely, let $f \in \Aut(N)$ such that \eqref{eq1} holds for all $x,y,z\in N$.  Note that setting $z=1$ in \eqref{eq1} gives 
\begin{equation}
x\cdot yf(x)=xy\cdot f(x).  
\label{commute1}
\end{equation}
Since $f$ is an automorphism and the above holds for all $x \in N$, replacing $x$ with $f^{-1}(x)$ yields 
\begin{equation}
f^{-1}(x)\cdot yx=f^{-1}(x)y\cdot x.
\label{commute2}
\end{equation}
Now, compute
\begin{alignat*}{3}
(x\cdot zf(x))f(y)
&=[x\cdot (y\cdot y^{-1}z)f(x)]f(y)
&&\stackrel{\eqref{eq1}}{=}[(xy)(y^{-1}z\cdot f(x))]f(y)\\
&=[(xy)(y^{-1}z\cdot f(x))](f(x)^{-1}\cdot f(xy))
&&\stackrel{\eqref{eq1}}{=}(xy)[(y^{-1}z\cdot f(x))f(x)^{-1}\cdot f(xy)]\\
&=(xy)\cdot (y^{-1}z)f(xy)
&&\stackrel{\eqref{eq1}}{=}(xy\cdot y^{-1})\cdot zf(xy)\\
&=x\cdot zf(xy).
\end{alignat*}
As before, replacing $x$ with $f^{-1}(x)$ yields 
\begin{equation}
(f^{-1}(x)\cdot zx)f(y)=f^{-1}(x)\cdot (z\cdot xf(y)).
\label{eq2}
\end{equation}
Hence, we have
\begin{alignat*}{3}
(x\cdot f^{-1}(y)z)(yf(x))
&\stackrel{\eqref{eq1}}{=}x[(f^{-1}(y)z\cdot y)f(x)]
&&\stackrel{\eqref{commute2}}{=}x[(f^{-1}(y)\cdot zy)f(x)]\\
&\stackrel{\eqref{eq2}}{=}x[f^{-1}(y)\cdot (z\cdot yf(x))].
\end{alignat*}
Multiplying by $f^{-1}(y)$ on the left gives
\begin{equation}
f^{-1}(y)[x\cdot f^{-1}(y)(z\cdot yf(x))]=f^{-1}(y)[(x\cdot f^{-1}(y)z)( yf(x))].
\label{eq3}
\end{equation}
Now, compute
\begin{align*}
f^{-1}(y)x\cdot z(yf(x))
&=f^{-1}(y)x\cdot zf(f^{-1}(y)x)\\
&=f^{-1}(y)x\cdot (x^{-1}\cdot xz)f(f^{-1}(y)x)\\
&\stackrel{\eqref{eq1}}{=}(f^{-1}(y)x\cdot x^{-1})(xz\cdot f(f^{-1}(y)x))\\
&=f^{-1}(y)[xz\cdot yf(x)].
\end{align*}
Replacing $z$ with $f^{-1}(y)z$ gives
\begin{equation}
(f^{-1}(y)x)\cdot (f^{-1}(y)z\cdot (yf(x))=f^{-1}(y)[(x\cdot f^{-1}(y)z)\cdot (yf(x))].
\label{eq4}
\end{equation}
Combining \eqref{eq3} and \eqref{eq4}, we have
\begin{align*}
f^{-1}(y)[x\cdot f^{-1}(y)(z\cdot yf(x))]
&\stackrel{\eqref{eq3}}{=}f^{-1}(y)[(x\cdot f^{-1}(y)z)\cdot( yf(x))]\\
&\stackrel{\eqref{eq4}}{=}(f^{-1}(y)x)[f^{-1}(y)z\cdot yf(x)].
\end{align*}
Replacing $y$ with $f(y)$ yields
\begin{equation}
y\cdot x(y\cdot zf(yx))=yx\cdot(yz\cdot f(yx)).
\label{eq5}
\end{equation}
Finally, we have
\begin{align*}
(xy\cdot x)z
&=(xy\cdot x)\cdot (zf(xy)^{-1})f(xy)\\
&\stackrel{\eqref{eq1}}{=}xy\cdot [(x\cdot zf(xy)^{-1})f(xy)]\\
&\stackrel{\eqref{eq5}}{=}x\cdot y(x\cdot (zf(xy)^{-1}\cdot f(xy)))\\
&=x(y\cdot xz).
\end{align*}
Therefore, $N$ is Moufang.
\end{proof}
\begin{theorem}
Let $N$ be a loop. Then $N$ is Moufang if and only if there exists some $f\in\Aut(N)$ which satisfies \eqref{eq1} for all $x,y,z\in N$.
\label{l3}
\end{theorem}
\begin{proof} As before, if $N$ is Moufang, then the trivial automorphism satisfies \eqref{eq1} for all $x,y,z\in N$.  Conversely, suppose there exists an $f\in \Aut(N)$ that satisfies \eqref{eq1} for all $x,y,z\in N$.  By Lemma \ref{l2}, it is enough to show that $N$ is an IP loop.  As before, we have
\begin{equation}\tag{\ref{commute1}}
x\cdot yf(x)=xy\cdot f(x),
\end{equation}
\begin{equation}\tag{\ref{commute2}}
f^{-1}(x)\cdot yx=f^{-1}(x)y\cdot x.
\end{equation}
Moreover, $f(x)=(x\cdot x\backslash 1)f(x)=x\cdot (x\backslash 1)f(x)$, implying $x\backslash f(x)=x\backslash 1 \cdot f(x)$.  Applying $f^{-1}$ yields
\[
f^{-1}(x)\backslash x = f^{-1}(x)\backslash 1 \cdot x.
\]
We compute
\begin{alignat*}{3}
\underbracket[.75 pt]{(f^{-1}(x)\backslash x)}(x\backslash 1)
&=[(f^{-1}(x)\backslash 1)\cdot x](x\backslash 1)
&&=[(f^{-1}(x)\backslash 1)\cdot x]\cdot f(f^{-1}(x)\backslash 1)\\
&\stackrel{\eqref{commute1}}{=}(f^{-1}(x)\backslash 1)(x\cdot f(f^{-1}(x)\backslash 1))
&&=(f^{-1}(x)\backslash 1)\cdot x(x\backslash 1)\\
&=f^{-1}(x)\backslash 1.
\end{alignat*}
Finally, we compute
\begin{alignat*}{3}
x
&=f^{-1}(x)(f^{-1}(x)\backslash 1)\cdot x
&&\stackrel{\eqref{commute2}}{=}f^{-1}(x)[\underbracket[.75 pt]{(f^{-1}(x)\backslash 1)}\cdot x]\\
&=f^{-1}(x)[(f^{-1}(x)\backslash x)(x\backslash 1)\cdot x]
&&=f^{-1}(x)[(f^{-1}(x)\backslash x)(x\backslash 1)\cdot f(f^{-1}(x))]\\
&\stackrel{\eqref{eq1}}{=}[f^{-1}(x)(f^{-1}(x)\backslash x)]\cdot [(x\backslash 1)(f(f^{-1}(x)))]
&&=x\cdot (x\backslash 1)x.
\end{alignat*}
This immediately implies $1/x=x\backslash 1=x^{-1}$.  

To show that $N$ is IP, observe that
\[
xy\cdot f(x)=x\cdot yf(x)=x[x^{-1}(x^{-1}\backslash y)\cdot f(x)]\stackrel{\eqref{eq1}}{=}xx^{-1}\cdot (x^{-1}\backslash y)f(x)=(x^{-1}\backslash y)f(x),
\]
which implies $xy=x^{-1}\backslash y$.  Hence, $x^{-1}\cdot xy=y$.  

To show that $yx\cdot x^{-1}=y$, note that $y^{-1}\cdot yx=x$ implies $x/(yx)=y^{-1}$.  Also note that \eqref{eq1} is equivalent to $[x\cdot (yz\cdot f(x))]/(zf(x))=xy$.  Using these two identities combined with $x^{-1}\cdot xy=y$, we have
\[
y^{-1}=f((xy)^{-1})/[yf((xy)^{-1})]=[(xy)^{-1}\cdot (xy\cdot f((xy)^{-1}))]/ [yf((xy)^{-1})]=(xy)^{-1}x.
\]
Hence, $xy\cdot y^{-1}=(xy)[(xy)^{-1}x]=x$.
\end{proof}
Theorem \ref{l3} yields an equivalent definition for a loop to be Moufang in terms of the behavior of one of its automorphisms.  In the same spirit, the next theorem gives a necessary and sufficient condition for a semidirect product of a loop and group to be Moufang.
\begin{theorem}
Let $N$ be a loop and $H$ a group.  Let $\phi:H\to \Sym(N)_{1}$ be a group homomorphism.  Then $G=N\rtimes_{\phi}H$ is a Moufang loop if and only if $\phi(H)\subseteq\Aut(N)$ and $\phi(h)$ satisfies \eqref{eq1} for all $h\in H$ and for all $x,y,z\in N$.
\label{t1}
\end{theorem}
\begin{proof}
Suppose $\phi(H)\subseteq\Aut(N)$ and $\phi(h)$ satisfies \eqref{eq1} for all $h \in H$ and for all $x,y,z\in N$, and let $(x,g),(y,h),(z,k)\in G$. Then, by hypothesis,
\begin{equation*}
 [x\phi(g)(y)]\cdot [\phi(gh)(z)\phi(ghk)(x)] \stackrel{\eqref{eq1}}{=} x[\phi(g)(y)\phi(gh)(z)\cdot \phi(ghk)(x)].
 \label{eq6}
\end{equation*}
Hence, 
\[
 ([x\phi(g)(y)]\cdot [\phi(gh)(z)\phi(ghk)(x)],ghkg)=(x[\phi(g)(y)\phi(gh)(z)\cdot \phi(ghk)(x)],ghkg),
\]
and thus
\[
[(x,g)(y,h)][(z,k)(x,g)]=(x,g)\cdot[(y,h)(z,k)\cdot(x,g)],
\]
which shows that $G$ is Moufang.

Conversely, if $G$ is Moufang, then by Lemma \ref{l1} we have $\phi(H)\subseteq\Aut(N)$.  Therefore 
\[
[(x,1_{H})(y,1_{H})][(z,h)(x,1_{H})]=(x,1_{H})[(y,1_{H})(z,h)\cdot (x,1_{H})],
\]
which is equivalent to
\[
(xy\cdot z\phi(h)(x),h)=(x(yz\cdot \phi(h)(x)),h). 
\]
Hence, $\phi(h)$ satisfies \eqref{eq1} for all $h\in H$, for all $x,y,z\in N$.
\end{proof}

Note that if $N=\texttt{MoufangLoop(12,1)}$, then it can be verified that there are no nontrivial automorphisms $f\in \Aut(N)$ which satisfy \eqref{eq1} for all $x,y,z\in N$.  Hence, for $H$ a group and $\phi:H\to \Aut(N)$ a nontrivial group homomorphism, then $G=N\rtimes_{\phi}H$ is not Moufang, which shows that the converse of Lemma \ref{l1} is false.  The next example illustrates that Theorem \ref{t1} can hold nontrivially.

\begin{example}
Let $N = \texttt{MoufangLoop(16,1)}$ and $H$ a group of order $2$.  Let $\phi:H\to \Aut(N)$ be a group homomorphism with $\phi(H)=\langle\texttt{(2,6)(3,7)(10,14)(11,15)}\rangle$.  It can be verified that $\phi(h)$ satisfies \eqref{eq1} for all $h\in H$ and for all $x,y,z\in N$.  Then $G=N\rtimes_{\phi}H$ is a Moufang loop of order $32$ ($\texttt{MoufangLoop(32,18)}$).
\end{example}

We end this section by showing that the class of IP loops is closed under semidirect products with groups acting as automorphisms.
\begin{proposition}
Let $N$ be a loop and $H$ a group.  Let $\phi:H\to\Aut(N)$ be a group homomorphism.  Then $G=N\rtimes_{\phi}H$ is an IP loop if an only if $N$ is an IP loop.
\end{proposition}
\begin{proof}
Clearly, if $G$ is IP, then $N$ is IP.
Conversely, if $N$ is IP, $x^{-1}$ and $g^{-1}$ exist for all $x\in N$ and $g\in G$.  It is easy to verify that $(x,g)^{-1}=(\phi(g^{-1})(x^{-1},g^{-1})$.  To show $G$ is IP, we compute
\begin{align*}
(x,g)^{-1}[(x,g)(y,h)]
&=(\phi(g^{-1})(x^{-1}),g^{-1})(x\cdot \phi(g)(y),gh)\\
&=(\phi(g^{-1})(x^{-1})\cdot\phi(g^{-1})[x\cdot \phi(g)(y)],g^{-1}gh)\\
&=((\phi(g^{-1})(x))^{-1}[\phi(g^{-1})(x)\cdot y],h)\\
&=(y,h).\\
[(y,h)(x,g)](x,g)^{-1}
&=(y\cdot\phi(h)(x),hg)(\phi(g^{-1})(x^{-1}),g^{-1})\\
&=([y\cdot \phi(h)(x)]\phi(hg)(\phi(g^{-1})(x^{-1})),hgg^{-1})\\
&=([y\cdot \phi(h)(x)](\phi(h)(x))^{-1},h)\\
&=(y,h).\qedhere
\end{align*}
\end{proof}
\section{Extensions and inverse property loops}
\label{conc}
We now turn our attention to extensions of loops as considered in \cite{gagola13}.
\begin{theorem}[\cite{gagola13}]
Suppose $G = NH$ is a Moufang loop, where $N$ is a normal subloop of $G$ and $H=\langle u \rangle$ is a finite cyclic subgroup of $G$ whose order is coprime to 3.  Then for any $xu^{m},yu^{n}\in G$ where $x,y\in N$, we have
\begin{equation}
(xu^{m})(yu^{n})=f^{\frac{2m+n}{3}}(f^{\frac{-2m-n}{3}}(x)f^{\frac{m-n}{3}}(y))u^{m+n},
\label{mult1}
\end{equation}
where 
\[
f:N\to N \qquad x\to uxu^{-1}
\]
is a semiautomorphism of $N$.  Moreover, $G$ is a group if and only if $N$ is a group and $f$ is an automorphism of $N$.
\label{gag}
\end{theorem}
Note that if $H=\langle u\rangle\simeq\mathbb{Z}_{2}$, then the multiplication in $G$ given by \eqref{mult1} is
\[
xu\cdot y=xf(y)\cdot u, \qquad x\cdot yu=f(f(x)f(y))\cdot u,\qquad xu\cdot yu =f(f(x)y) ,
\]
for all $x,y\in N$.  Hence, multiplication in such a loop $G$ is easily seen to be equivalent to the multiplication given in \cite{chein2}, Lemma $1$, pg $21$.  These types of extensions have been well-studied, see \cite{chein1,chein2,greer2}.
\begin{example}
A loop $Q$ is said to be a \emph{semiautomorphic, inverse property loop} (or just semiautomorphic IP loop) if
\begin{enumerate}
\item $Q$ is flexible; that is, $(xy)x=x(yx)$ for all $x,y\in Q$;
\item $Q$ is an IP loop; and
\item Every inner mapping is a semiautomorphism.
\end{enumerate}
It is known that Moufang loops are semiautomorphic IP loops and that semiautomorphic IP loops are diassociative loops \cite{KKP}.  Let $N$ be a semiautomorphic IP loop and $H=\langle u \rangle$ have order $2$, and let $\phi:H\to \Semi(N)$ be the group homomorphism such that $\phi(u)(x) = x^{-1}$ for all $x \in N$. Then $G=(N,H,\phi)$ with multiplication extended by \eqref{mult1} is a semiautomorphic IP loop (see \cite{greer2}), but not necessarily Moufang, since it is shown in \cite{chein2} that $G$ is Moufang if and only if $N$ is a group.
\label{e2}
\end{example}

Here, we consider the \emph{external} viewpoint.  That is, the external extension of a cyclic group $H=\langle u \rangle$ by a loop $N$ with $\phi:H\to \Semi(N)$ a group homomorphism. We extend the multiplication from $N$ to $G=(N,H,\phi)$ as 
\begin{equation}
(x,u^{m})(y,u^{n})=(f^{\frac{2m+n}{3}}(f^{\frac{-2m-n}{3}}(x)f^{\frac{m-n}{3}}(y)),u^{m+n}),
\label{mult2}
\end{equation}
where $f(x)=\phi(u)(x)$.  As the previous example shows, if $N$ is Moufang, $G$ need not be Moufang.  However, as the next theorem illustrates, the class of IP loops is closed under such extensions.
\begin{theorem}
Let $N$ be a loop, $H=\langle u \rangle$ a cyclic group, and let $\phi:H\to \Semi(N)$ be a group homomorphism.  If $|\phi(u)|$ is coprime to 3, define an extension $G = (N,H,\phi)$ with multiplication given by
\begin{equation*}
(x,u^{m})(y,u^{n})=(f^{\frac{2m+n}{3}}(f^{\frac{-2m-n}{3}}(x)f^{\frac{m-n}{3}}(y)),u^{m+n}),
\end{equation*}
where $f(x) =\phi(u)(x)$. Then $G$ is an IP loop if and only if $N$ is an IP loop.
\label{t2}
\end{theorem}
\begin{proof}
Clearly, $N$ is an IP loop if $G$ is an IP loop. Now, suppose $N$ is an IP loop. Note that since the order of $f$ in $\Semi(N)$ is coprime to $3$, $f$ must have a unique cubed root, denoted by $f^{\frac{1}{3}}$.  It is clear that $G$ is a loop with identity $(1,u^{0})$.  For $(x,u^{m}) \in G$, consider the element $(f^{-m}(x^{-1}),u^{-m})$.  Note that this is well-defined, since $N$ is an IP loop and $H$ is a group.  Then,
\begin{align*}
(x,u^{m})(f^{-m}(x^{-1}),u^{-m})
&=(f^{\frac{2m-m}{3}}(f^{\frac{-2m+m}{3}}(x)f^{\frac{m+m}{3}}(f^{-m}(x^{-1}))),u^{0})\\
&=(f^{\frac{m}{3}}(f^{\frac{-m}{3}}(x)f^{\frac{-m}{3}}(x^{-1})),u^{0})\\
&=(f^{\frac{m}{3}}(f^{\frac{-m}{3}}(x)(f^{\frac{-m}{3}}(x))^{-1}),u^{0})\\
&=(f^{\frac{m}{3}}(1),u^{0})\\
&=(1,u^{0})\\
&=(f^{\frac{-m}{3}}(1),u^{0})\\
&=(f^{\frac{-m}{3}}((f^{\frac{-2m}{3}}(x))^{-1}f^{\frac{-2m}{3}}(x)),u^{0})\\
&=(f^{\frac{-m}{3}}(f^{\frac{-2m}{3}}(x^{-1})f^{\frac{-2m}{3}}(x)),u^{0})\\
&=(f^{\frac{-2m+m}{3}}(f^{\frac{2m-m}{3}}(f^{-m}(x^{-1})f^{\frac{-m-m}{3}}(x))),u^{0})\\
&=(f^{-m}(x^{-1}),u^{-m})(x,u^{m}).
\end{align*}
Hence, $G$ has two sided inverses for all $(x,u^{m})\in G$, namely $(x,u^{m})^{-1}=(f^{-m}(x^{-1}),u^{-m})$.  To show that $G$ is an IP loop, we compute
\begin{align*}
(x,u^{m})[(f^{-m}(x^{-1}),&u^{-m})(y,u^{n})]\\ 
&=(x,u^{m})(f^{\frac{-2m+n}{3}}(f^{\frac{2m-n}{3}}(\A)f^{\frac{-m-n}{3}}(y)),u^{-m+n})\\
&=(f^{\frac{m+n}{3}}(f^{\frac{-m-n}{3}}(x)   [f^{\frac{2m-n}{3}}(f^{\frac{-2m+n}{3}}(f^{\frac{2m-n}{3}}(\A)f^{\frac{-m-n}{3}}(y)))]),u^{n})\\
&=(f^{\frac{m+n}{3}}(f^{\frac{-m-n}{3}}(x)   [(f^{\frac{2m-n}{3}}(\A)f^{\frac{-m-n}{3}}(y))]),u^{n})\\
&=(f^{\frac{m+n}{3}}(f^{\frac{-m-n}{3}}(x)[f^{\frac{-m-n}{3}}(x^{-1})f^{\frac{-m-n}{3}}(y)]),u^{n})\\
&=(f^{\frac{m+n}{3}}(f^{\frac{-m-n}{3}}(x)[(f^{\frac{-m-n}{3}}(x))^{-1}f^{\frac{-m-n}{3}}(y)]),u^{n})\\
&=(f^{\frac{m+n}{3}}(f^{\frac{-m-n}{3}}(y)),u^{n})\\
&=(y,u^{n}).\\
&\\
[(y,u^{n})(\A,&u^{-m})](x,u^{m})\\
&=(f^{\frac{2n-m}{3}}[f^{\frac{-2n+m}{3}}(y)f^{\frac{n+m}{3}}(\A)],u^{n-m})(x,u^{m})\\
&=(f^{\frac{2n-m}{3}}[f^{\frac{-2n+m}{3}}(f^{\frac{2n-m}{3}}[f^{\frac{-2n+m}{3}}(y)f^{\frac{n+m}{3}}(\A)])f^{\frac{n-2m}{3}}(x)],u^{n})\\
&=(f^{\frac{2n-m}{3}}(([f^{\frac{-2n+m}{3}}(y)f^{\frac{n+m}{3}}(\A)])f^{\frac{n-2m}{3}}(x)),u^{n})\\
&=(f^{\frac{2n-m}{3}}([f^{\frac{-2n+m}{3}}(y)f^{\frac{n-2m}{3}}(x^{-1})]f^{\frac{n-2m}{3}}(x)),u^{n})\\
&=(f^{\frac{2n-m}{3}}([f^{\frac{-2n+m}{3}}(y)(f^{\frac{n-2m}{3}}(x))^{-1}]f^{\frac{n-2m}{3}}(x)),u^{n})\\
&=(f^{\frac{2n-m}{3}}(f^{\frac{-2n+m}{3}}(y),u^{n}))\\
&=(y,u^{n}). \qedhere
\end{align*}
\end{proof}
The next two examples illustrate that if certain stronger conditions are imposed on $N$, then an extension as in Theorem \ref{t2} need not inherit a similar structure.
\begin{example}
Let $N$ be the quaternion group of order $8$ (GAP \texttt{SmallGroup(8,4)}) and $H$ a group of order $2$.  Let $\phi:H\to \Semi(N)$ a group homomorphism with $\phi(H)=\langle\texttt{(2,5)(6,8)}\rangle$.  Then the extension $G=(N,H,\phi)$ as in Theorem \ref{t2} is a power-associative (the subloop generated by $x$ is associative for all $x\in G$), IP loop.  It can be verified that $G$ is neither flexible nor diassociative, although it is clear that $N$ has both properties.
\end{example}

\emph{Steiner loops}, which arise from Steiner triple systems in combinatorics, are loops satisfying the identities $xy=yx$, $x(yx)=y$.  In particular, Steiner loops are semiautomorphic IP loops.
\begin{example}
Let $N$ be a Steiner loop of order $8$ (\texttt{SteinerLoop(8,1)}) and $H=\langle u \rangle $ a group of order $4$.  Let $\phi:H\to \Semi(N)$ be a group homomorphism with $\phi(H)=\langle\texttt{(2,4,8,3)}\rangle$.  Then the extension $G=(N,H,\phi)$ as in Theorem \ref{t2} is an IP loop of order $32$ that is neither power-associative nor flexible, although it is clear that $N$ has both properties.
\end{example}
\begin{corollary}
Let $N$ be a Moufang loop, $H=\langle u \rangle$ a cyclic group, and let $\phi:H\to \Semi(N)$ be a group homomorphism.  If $|\phi(u)|$ is coprime to 3, define an extension $G = (N,H,\phi)$ with multiplication given by
\begin{equation*}
(x,u^{m})(y,u^{n})=(f^{\frac{2m+n}{3}}(f^{\frac{-2m-n}{3}}(x)f^{\frac{m-n}{3}}(y)),u^{m+n}),
\end{equation*}
where $f(x) =\phi(u)(x)$. Then $G$ is an IP loop.
\label{cor2}
\end{corollary}
\begin{proof}
Since Moufang loops are IP loops, the desired result follows from Theorem \ref{t2}.
\end{proof}

\begin{acknowledgment}
Some investigations in this paper were assisted by the automated deduction tool, \textsc{Prover9}, and the finite model builder, \textsc{Mace4}, both developed by McCune \cite{mccune09}.  Similarly, all presented examples were found and verified using the GAP system \cite{GAP} together with the LOOPS package \cite{GAPNV}.  
\end{acknowledgment}

\bibliographystyle{amsplain}

\end{document}